\documentclass[journal]{IEEEtran}
\usepackage{ifpdf}
\usepackage{cite}
\ifCLASSINFOpdf
  \usepackage[pdftex]{graphicx}
  \DeclareGraphicsExtensions{.pdf,.jpeg,.png}
\else
  \usepackage[dvips]{graphicx}
  \DeclareGraphicsExtensions{.eps}
\fi
\usepackage{amsmath}
\usepackage{algorithmic}
\usepackage{array}
\ifCLASSOPTIONcompsoc
  \usepackage[caption=false,font=normalsize,labelfont=sf,textfont=sf]{subfig}
\else
  \usepackage[caption=false,font=footnotesize]{subfig}
\fi
\usepackage{stfloats}
\usepackage{url}

\usepackage{amssymb}
\usepackage{amsthm}
\usepackage{soul}
\usepackage{hyperref}
\usepackage{inputenc}
\usepackage{caption}
\usepackage{booktabs}
\usepackage{latexsym}
\usepackage{picinpar}
\usepackage{color, soul, colortbl}
\usepackage{wrapfig}
\usepackage{mdframed}
\usepackage{amsmath,amsfonts,mathrsfs,amssymb,bm,mathtools,mathabx}
\usepackage{xspace}
\usepackage[pdftex,dvipsnames]{xcolor}  
\usepackage{todonotes}
\usepackage[ruled,linesnumbered,noresetcount,vlined]{algorithm2e}
\usepackage{hyperref} 
\hypersetup{
colorlinks=true, 
breaklinks=true,
urlcolor= cyan, 
linkcolor= cyan,
citecolor=magenta, 
}
\usepackage{float}
\floatstyle{ruled}
\newfloat{model}{thp}{lop}
\floatname{model}{Model}

\usepackage{multicol}

\usepackage{booktabs}
\usepackage{longtable}
\usepackage{xcolor}
\usepackage{enumitem}
\usepackage[export]{adjustbox}
\usepackage{cleveref}

\def\thmspace{0.2em}
\newtheorem{theorem}{\hspace{\thmspace}{\bf Theorem}\!}

\newtheorem{definition}{\hspace{\thmspace}{\bf Definition}\!}

\newtheorem{lemma}{\hspace{\thmspace}{\bf Lemma}\!}

\DeclareMathOperator*{\argmin}{argmin}

\newcommand{\cC}{\mathcal{C}}
\newcommand{\cD}{\mathcal{D}}

\newcommand{\cJ}{\mathcal{J}}

\newcommand{\cM}{\mathcal{M}}
\newcommand{\cN}{\mathcal{N}}
\newcommand{\cO}{\mathcal{O}}
\newcommand{\cP}{\mathcal{P}}
\newcommand{\cR}{\mathcal{R}}
\newcommand{\Lap}{{\text{Lap}}}

\newcommand{\etal}{\textit{et al.}}

\begin{document}
\title{Bilevel Optimization for Differentially Private Optimization in Energy Systems}
%
%
%

\author{Terrence W.K.~Mak, %
        Ferdinando~Fioretto, %
        and~Pascal Van~Hentenryck
      \thanks{T.W.K. Mak and P. Van Hentenryck 
	are affiliated with 
	the School of Industrial and Systems Engineering, 
	Georgia Institute of Technology, Atlanta, GA 30332. 
	F.~Fioretto is affiliated with the Electrical Engineering and Computer Science Department, Syracuse University, Syracuse, NY 13244, USA.
	E-mail contacts: wmak@gatech.edu, ffiorett@syr.edu, pvh@isye.gatech.edu.
	}}

\markboth{}%
{Mak \MakeLowercase{\textit{et al.}}}

\maketitle

\begin{abstract}
  This paper studies how to apply differential privacy to constrained
  optimization problems whose inputs are sensitive. This task raises
  significant challenges since random perturbations of the input data
  often render the constrained optimization problem infeasible or
  change significantly the nature of its optimal solutions. To address
  this difficulty, this paper proposes a bilevel optimization model
  that can be used as a post-processing step: It redistributes the
  noise introduced by a differentially private mechanism optimally
  while restoring feasibility and near-optimality.  The paper shows
  that, under a natural assumption, this bilevel model can be solved
  efficiently for real-life large-scale nonlinear nonconvex
  optimization problems with sensitive customer data.  The
  experimental results demonstrate the accuracy of the
  privacy-preserving mechanism and showcases significant benefits
  compared to standard approaches.
\end{abstract}


\IEEEpeerreviewmaketitle

\section{Introduction}
\label{sec:introduction}

Differential Privacy (DP) \cite{dwork:06} is a robust framework used 
to measure and bound the privacy risks in computations over 
datasets: It has been successfully applied to numerous applications including
histogram queries \cite{li2010optimizing}, census surveys \cite{abowd2018us,Fioretto:cp-19}, linear regression \cite{chaudhuri:2011} and deep learning \cite{abadi2016deep} to name but a few
examples. In general, DP mechanisms ensure privacy by introducing
calibrated noise to the outputs or the objective of
computations. However, its applications to large-scale, complex
constrained optimization problems have been sparse.

This paper considers parametric optimization problems of the form
\begin{flalign}
{\cal O}(d) = \min_x f(x) \mbox{ s.t. } g(x,d) \geq 0, x \geq 0,
\tag{OPT} \label{OPT} 
\end{flalign} 
where $x$ is a vector of decision variables, $d$ is a real valued vector of problem inputs, and $g(x,d)$ is an abstract function capturing all the constraints on $x$ and $d$. Without loss of generality, the paper assumes that $x$ is non-negative.
Given a vector $d^o$ of sensitive data, the task is to find a
differentially private vector $d^*$ such that $d^* \approx d^o$ and
${\cal O}(d^*) \approx {\cal O}(d^o)$. Effective solutions to this
task are useful in various settings, including the generation of
differentially private test cases for \eqref{OPT} or in sequential
coordination problems, e.g. sequential markets, 
in which agents need to exchange private versions of their data to 
solve ${\cal O}$. It is possible to use traditional
differential privacy techniques (e.g., the ubiquitous Laplace or the Exponential mechanisms) to obtain a private vector $\tilde{d}$ such that $\tilde{d}
\approx d^o$. However, in general, the optimization problem
\eqref{OPT} may not admit any feasible solution for input $\tilde{d}$ or, its objective value ${\cal O}(\tilde{d})$ can be far from ${\cal
  O}(d^o)$. 

This paper aims at remedying this fundamental limitation. Given private versions $\tilde{d}$ of $d^o$ and $\tilde{f}$ of ${\cal O}(d^o)$, it proposes
a bilevel optimization model that leverages the post-processing
immunity of differential privacy to produce a new private vector $d^*$,
based on $\tilde{d}$,
such that ${\cal O}(d^*) \approx {\cal O}(d^o)$. The paper also
presents an algorithm that solves the bilevel model optimally under a
natural monotonicity assumption. The effectiveness of the approach is
demonstrated on large-scale case studies in electrical and gas
networks where customer demands are sensitive, including nonlinear
nonconvex benchmarks with more than $10^4$ variables. The paper
generalizes prior approaches (e.g. \cite{mak19privacy}) 
designed to release benchmarks in energy
systems. Its main contributions are as follows:
\begin{enumerate}[noitemsep, nolistsep, leftmargin=12pt, itemindent=12pt]
\item It presents a general mathematical framework, that relies on bilevel optimization, to obfuscate data $d^o$ of parametric optimization problems;  
\item It proposes an efficient algorithm that solves the bilevel optimization model under a natural monotonicity assumption;
\item It demonstrates the effectiveness of the proposed method on two case studies on energy systems and empirically validates the monotonicity assumption on these case studies.
\item Finally, it demonstrates that optimal solutions to the bilevel model can produce significant improvements in accuracy compared to its relaxations. 
\end{enumerate}

\section{Related Work}
\label{sec:related}
The literature on theoretical results of Differential Privacy (DP) is huge (e.g., \cite{dwork:13,vadhan:17}).  The literature on DP applied to energy systems includes considerably fewer efforts.  {\'A}cs and Castelluccia \cite{acs:11} exploit a direct application of the Laplace mechanism to hide user participation in smart meter datasets, achieving $\epsilon$-DP.  Zhao et al.~\cite{zhao:14} study a DP schema that exploits the ability of households to charge and discharge a battery to hide the real energy consumption of their appliances.  Liao et al.~\cite{liao:17} introduce Di-PriDA, a privacy-preserving mechanism for appliance-level peak-time load balancing control in the smart grid, aimed at masking the consumption of top-k appliances of a household. 
Halder et al.~\cite{halder17architecture} propose an architecture for privacy-preserving thermal inertial load management as a service provided by load-serving entities.
Zhou \etal~\cite{zhou2019} later also present a
particularly interesting relaxed notion of a (stronger) monotonicity
property for a DC-OPF operator and shows how to use it to compute the
operator sensitivity.
This enables a characterization of the network sensitivity, which is useful to preserve the privacy of \emph{monotonic networks}.
A different line of work, conducted by Karapetyan et al.~\cite{karapetyan:17} quantifies empirically the trade-off between privacy and utility in demand  response systems. The authors analyze the effects of a simple Laplace  mechanism on the objective value of the demand response optimization  problem. 
A DP schema that uses constrained post-processing was recently introduced by Fioretto et al.~\cite{fioretto:AAMAS-18} and adopted to release private mobility data.  

There are also related work on privacy-preserving implementations for the 
Alternating Direction Method of Multipliers (ADMM) algorithm. Zhang et al.~\cite{zhang2016dynamic} proposed a  version of the ADMM algorithm for privacy-preserving empirical risk minimization problems, a class of convex problems used for regression and classification tasks. 
Huang et al.~\cite{huang2019dp} proposed an approach that combines an approximate augmented Lagrangian function with time-varying Gaussian noise for general objective functions. 
Ding et al.~\cite{ding2019optimal} proposed P-ADMM, to provide guarantees within a \emph{relaxed} model of differential privacy (called zero-concentrated DP).
Finally, our recent work \cite{mak20privacy} proposed an ADMM formulation to obfuscate 
power system data for the data releasing process.

In contrast to previous work, this paper proposes a general solution for the releasing differentially private inputs of constrained optimization problems while guaranteeing solution feasibility and bounded distance to optimality. 
This work generalizes previous results \cite{mak19privacy} that apply to power systems.
In addition, the paper extracts and reformulates the monotonicity property and provide proofs on the optimality of the
algorithm under the monotonicity assumption. 
Experimental results on power and gas networks 
are also provided as empirical evidence 
for the monotonicity assumption.  


\begin{table*}[t]
\begin{center}
\caption{Nomenclature.\label{tab:notation}}
	\resizebox{0.75\linewidth}{!}
	{
	\begin{tabular}{l l  l l}
    \toprule
		$d \in {\cal N}$ & Data vector $d$ from collection ${\cal N}$& $\alpha$ & Indistinguishability parameter \\
		$\epsilon$ & Privacy parameter & $f$& Cost/Objective value \\
		$\beta$ & Cost acceptance threshold & ${\cal S}$ & Set of optimal solutions \\
		${\cal F}$ & Set of feasible solution & ${\cal O}$ & Parametric optimization problem (OPT) \\
		$v^o / \tilde{v}$ & The original / obfuscated value $v$ & $v^*$& Optimally obfuscated value $v$ \\
		\bottomrule
  	\end{tabular}
  	}
	\end{center}
 \end{table*}

\section{Preliminaries}
\label{sec:dp}

The traditional definition of differential privacy \cite{dwork:06}
aims at protecting the potential participation of an individual in a
computation. For optimization problems however, the participants are
typically known and the input is a vector $d = \langle d_1, \ldots,
d_m \rangle$, where $d_i$ represents a \emph{sensitive quantity}
associated with participant $i$. For instance, $d_i$ may represent the
energy consumption of an industrial customer in an electrical
transmission system. This privacy notion is best captured by the
$\alpha$-indistinguishability framework proposed by
Chatzikokolakis \etal~\cite{chatzikokolakis2013broadening} which protects the sensitive
data of each individual up to some measurable quantity $\alpha > 0$. As a
result, this paper uses an \emph{adjacency relation} $\sim_\alpha$
for input vectors defined as follows:
\begin{equation*} 
\label{eq:adj_rel} 
    d \sim_\alpha d' \Leftrightarrow \exists i
    \textrm{~s.t.~} | d_i - {d'}_{i} | \leq \alpha \;\land\;
    d_j = {d'}_{j}, \forall j \neq i,
\end{equation*} 
where $d$ and ${d'}$ are input vectors to \eqref{OPT} and $\alpha > 0$
is a positive real value. This adjacency relation is used to protect
\emph{an individual value} $d_i$ up to privacy level $\alpha$ even if
an attacker acquires information about all other inputs $d_j$ ($j \neq
i$).

\begin{definition}[Differential Privacy]
  \label{eq:dp_def}
  Let $\alpha \!>\! 0$.  A randomized mechanism $\cM \!:\! \cD \!\to\!\cR$ 
  with domain $\cD$ and range $\cR$ is ($\epsilon,\alpha$)-indistinguishable if, for any output response $O \subseteq \cR$ and any two
  \emph{adjacent} input vectors $d$ and $d'$ such that  $d \sim_\alpha d'$,
\begin{equation*}
  Pr[\cM(d) \in O] \leq e^\epsilon Pr[\cM({d'}) \in O].
\end{equation*}
\end{definition}

\noindent 
Parameter $\epsilon \!\geq\! 0$ controls the level of \emph{privacy}, with small values denoting strong privacy, while $\alpha$ controls the level of \emph{indistinguishability}.
For notational simplicity, this paper assumes that $\epsilon$ is fixed
to a constant and refers to mechanisms satisfying the definition above
as $\alpha$-indistinguishable.


The post-processing immunity of DP \cite{dwork:13} guarantees that a
private dataset remains private even when subjected to arbitrary
subsequent computations.

\begin{theorem}[Post-Processing Immunity] 
\label{th:postprocessing} 
Let $\cM$ be an $\alpha$-indistinguishable mechanism and $g$ be
a data-independent mapping from the set of possible output sequences to 
an arbitrary set. Then, $g \circ \cM$ is $\alpha$-indistinguishable.
\end{theorem}

A real function $f$ over a vector $d$ can be made indistinguishable by
injecting carefully calibrated noise to its output.  The amount of
noise to inject depends on the \emph{sensitivity} $\Delta_f$ of $f$
defined as $ \Delta_f = \max_{d \sim_\alpha {d'}} \left\| f(d) -
f({d'})\right\|_1.$ For instance, querying a customer load from a
dataset $d$ corresponds to an identity query whose sensitivity is
$\alpha$. The Laplace mechanism achieves $\alpha$-indistinguishability
by returning the randomized output $f(d)+ z$, where $z$ is drawn from
the Laplace distribution
$\textrm{\Lap}\left({\Delta_f}/{\epsilon}\right)$
\cite{chatzikokolakis2013broadening}.

Table~\ref{tab:notation} summarizes the notation  adopted throughout the paper.

\section{Differentially Private Optimization}
\subsection{Problem Definition}
\label{section:pb}

Consider the parametric optimization problem \eqref{OPT}, a sensitive
vector $d^o$, an $\alpha$-indistinguishable version $\tilde{d}$ of
$d^o$, and an approximation $\tilde{f}$ of $f^o = {\cal O}(d^o)$. For
instance, $\tilde{d}$ can be obtained by applying the Laplace
mechanism on identity queries on all $d_i$; $\tilde{f}$ can be a
private version of ${\cal O}(d^o)$, or the value ${\cal O}(d^o)$
itself if it is public
\footnote{
Previous work~\cite{dwork2011firm,dwork:13} has shown that 
even if an attacker has a hold on some publicly available information about the data the privacy loss on the data set 
is still bounded by the differential privacy mechanism. 
}
which is typically the case when the optimization
is a market-clearing mechanism, or an approximation of ${\cal
  O}(d^o)$ obtained using public information only (e.g., a public
forecast of $d^o$). The paper simply assumes that $|{\cal O}(d^o) -
\tilde{f}| \leq \beta^o$ for some 
user defined value $\beta^o > 0$, which is not
restrictive.
Note that the assumption obviously holds when $f^o$ is public.
The
goal is to find a vector $d^*$ using only $\tilde{d}$, $\tilde{f}$,
and the definition of \eqref{OPT} such that
\[
d^* \approx \tilde{d} \mbox{ and } {\cal O}(d^*) \approx \tilde{f}.
\]
Observe that, by Theorem \ref{th:postprocessing}, $d^*$ will be
$\alpha$-indistinguishable. It will be close to $d^o$ if $\tilde{d}$ is
close to $d^o$. Moreover, \eqref{OPT} is feasible for $d^*$ and ${\cal
  O}(d^*)$ will be close to ${\cal O}(d^o)$ if $\tilde{f}$ is. The paper
uses ${\cal S}(d)$ to denote the set of optimal solutions to
\eqref{OPT}, i.e.,
\begin{flalign*}
      {\cal S}(d) = \argmin_x f(x) \mbox{ s.t. } g(x,d) \geq 0, x \geq 0,
\end{flalign*}
and ${\cal F}(d)$ to denote
the set of feasible solutions, i.e., \begin{flalign*} {\cal F}(d) = \{
  x \mid g(x,d) \geq 0, x \geq 0 \}.
\end{flalign*}
The paper assumes that ${\cal F}(d^o)$ is not empty.

\subsection{The Bilevel Optimization Model}

The problem defined in Section \ref{section:pb} can be tackled by a
bilevel optimization model \ref{eq:BL}, i.e.,
\begin{subequations} 
  \makeatletter
  \def\@currentlabel{BL}
  \makeatother
  \label{eq:BL}
  \renewcommand{\theequation}{BL\arabic{equation}}
  \begin{alignat}{2}
  d^* = \argmin_{d} & \;\; \| d - \tilde{d} \|^2_2 \tag{BL1} \label{BL1} 
  \\
  \text{s.t.}~~
        & |{\cal O}(d) - \tilde{f}| \leq \beta. \tag{BL2} \label{BL2}
\end{alignat}
\end{subequations}
Its output is a private $d^*$ whose L$^2$-distance to $\tilde{d}$ is
minimized and whose value ${\cal O}(d^*)$ is in the interval
$[\tilde{f} - \beta,\tilde{f} + \beta]$ for a parameter $\beta \geq \beta^o$.
To make the bilevel nature more explicit, the above can be reformulated as
\begin{align}
\min_{d,x^*} & \;\; \| d - \tilde{d} \|^2_2 \tag{BL1'} \label{BL1'} \\ 
\text{s.t.}~~ & 
        |f(x^*) - \tilde{f}| \leq \beta \tag{BL2'} \label{BL2'} \\ 
        &  x^* = \argmin_{x \geq 0} f(x) \mbox{ s.t. }  g(x,d) \geq 0. \tag{Follower} \label{BL3}
\end{align}

\noindent
Note that $d^o$ is a feasible
solution to \eqref{eq:BL},
but not necessarily optimal.
The set of optimal solutions to \eqref{eq:BL} is denoted by ${\cal S}^{BL}$ and the set of feasible solutions by ${\cal F}^{BL}$.
If the private data $\tilde{d}$ satisfies \eqref{BL2} and 
does not require any post-processing,  
it is returned by model \ref{eq:BL}
(i.e. $d^* = \tilde{d}$).
%

\begin{theorem}
\label{th:bound}
$d^* \in {\cal S}^{\it{BL}}$ implies
$
	\| d^* - d^o \|_{2} \leq 2 \| \tilde{d} - d^o \|_{2}.
$
\end{theorem}
\begin{proof}
Theorem \ref{th:bound} generalizes a prior result from
\cite{fioretto:18b,mak19privacy}.
\begin{align*}
\| d^* - d^o \|_{2} & \leq \| d^* - \tilde{d} \|_{2} + \| \tilde{d} - d^o \|_{2} \label{eq:p2}\\
	    &\leq 2 \|\tilde{d} - d^o\|_{2}.
\end{align*}
where the first inequality follows from the triangle inequality on norms and the second inequality follows from
$
\| d^* - d^o \|_{2} \leq \|d^o - \tilde{d}\|_{2}
$ 
by optimality of $d^*$ and $d^o \in {\cal F}^{BL}$.
\end{proof}

 It implies that, when a Laplace mechanism
produces $\tilde{d}$, a solution $d^* \in {\cal S}^{\it{BL}}$ is no
more than a factor of 2 away from optimality since the Laplace
mechanism is optimal for identity queries
\cite{koufogiannis:15}. In other words, {\em \eqref{eq:BL} restores feasibility
  and near-optimality at a constant cost in accuracy}. In
practice, as shown in Section \ref{sec:experiments}, $d^*$ is typically closer to $d^o$ than $\tilde{d}$ is.

\subsection{High Point Relaxation (HPR)}
Bilevel optimization is computationally challenging. It is strongly
NP-hard \cite{hansen:92} and even determining the optimality of
a solution \cite{vicente:94} is NP-hard.
The High Point Relaxation \eqref{eq:HPR}, defined as 
\begin{subequations}
  \makeatletter
  \def\@currentlabel{HPR}
  \makeatother
  \label{eq:HPR}
  \renewcommand{\theequation}{HPR\arabic{equation}}
\begin{alignat}{2}
d^h = \argmin_{d,x \geq 0} & \;\; \| d - \tilde{d} \|^2_2 \tag{HPR1} \\
\text{s.t.}~~
	&  g(x,d) \geq 0  \label{eq:HPR2} \tag{HPR2} \\
        & |f(x) - \tilde{f} | \leq \beta \tag{HPR3} 
\end{alignat}
\end{subequations}
is an important tool in bilevel optimization.
Due to the nature of relaxation, ${\cal O}(d^h) \!\leq\! {\cal O}(d^*)$. 
Theorem \ref{th:bound} also holds for \eqref{eq:HPR},
i.e., $\| d^h - d^o \|_{2} \!\leq\! 2 \| \tilde{d} - d^o \|_{2}$.  
The set
of optimal solutions to \eqref{eq:HPR} is denoted by ${\cal S}^{\it{HPR}}$ and
its set of feasible solutions by ${\cal F}^{\it{HPR}}$.  For
simplicity, $(d,x) \in {\cal S}^{\it{P}}$ and $d \in {\cal
  S}^{\it{P}}$ are both used to denote an optimal solution to a
problem (P) and its projection to $d$.


\begin{figure}[t]
	\centering
	 \includegraphics[width=.45\textwidth]{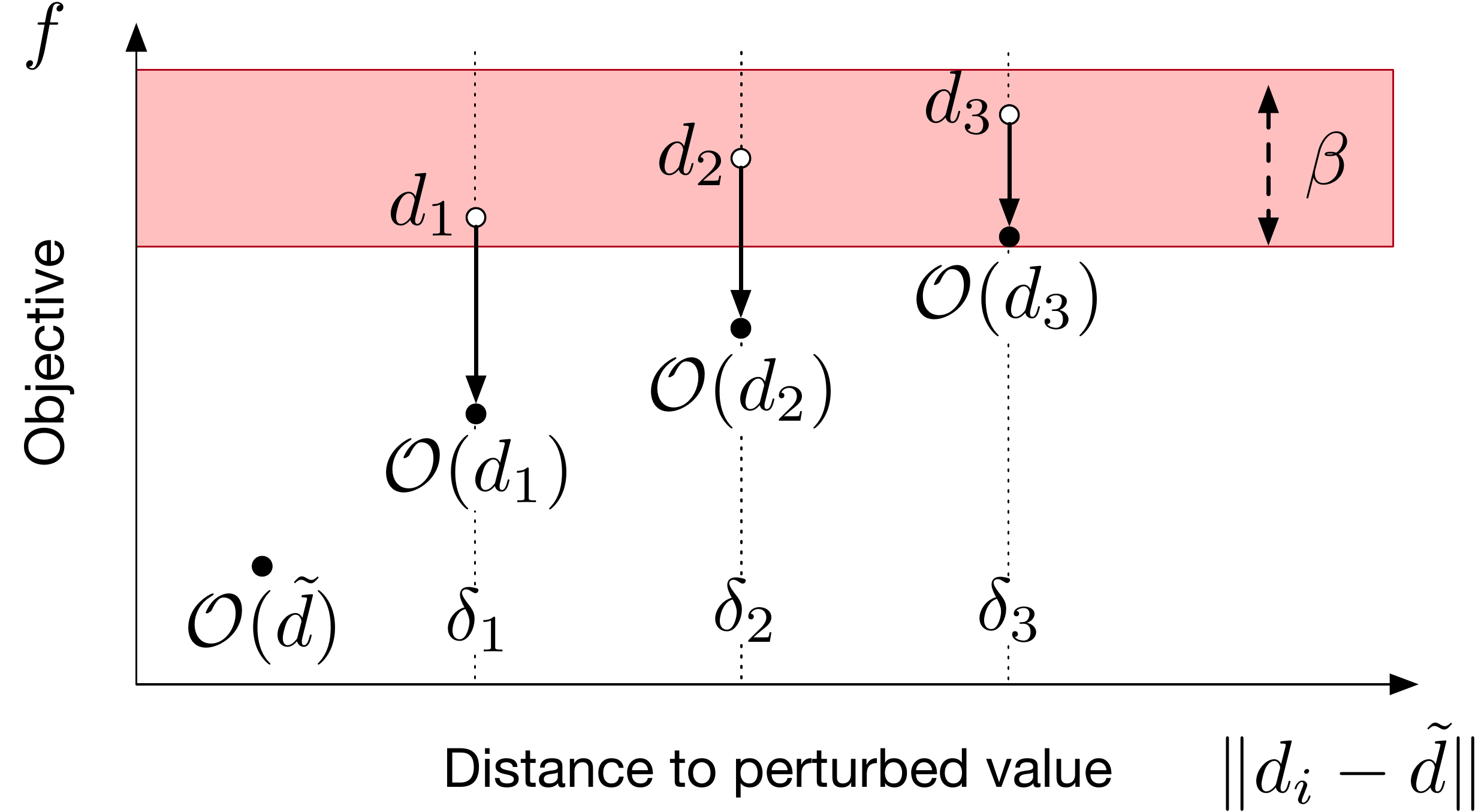}
	 \caption{Illustrating the Intuition Underlying ${\cal O}^{\uparrow}$.
	 }%
	\label{fig:push_exp}
\end{figure}

\subsection{Solving the Bilevel Model}

While bilevel optimization is, in general, computationally challenging,
Model \ref{eq:BL} presents a substantial structure: \eqref{BL3} minimizes $f$
but \eqref{BL2'} restrains it to be in a tight interval and \eqref{BL1'}
keeps the potential solution close to $\tilde{d}$.  The proposed
solution technique leverages these observations and an additional
insight derived from practical applications.

\paragraph{Intuition}

Consider a pair $(\bar{d},\bar{x})$ such that $\bar{x} \geq 0$,
$g(\bar{x},\bar{d}) \geq 0$, and $|f(\bar{x}) - \tilde{f}|  \leq
\beta$. Note that $\bar{x}$ is not necessarily an optimal solution 
(i.e., a minimizer) to Problem \eqref{BL3}. 
However, if the optimal objective value ${\cal O}(\bar{d})$ is such that ${\cal O}(\bar{d}) \geq \tilde{f} - \beta$,
then $\bar{d} \in {\cal F}^{BL}$. The interesting case is when ${\cal
  O}( \bar{d}) < \tilde{f} - \beta$. The proposed solution method
recognizes that, in many optimization problems with sensitive data,
$\bar{d}_i$ represents some data about participant $i$, such as her
electricity consumption. Assume, for instance, that $\bar{d}$
represents the customer demands (the reasoning is similar if the data
represents prices and reversed if the data represents production
capabilities).  By increasing $\bar{d}_i$, the value ${\cal O}(
\bar{d})$ is expected to rise as well. The solution technique exploits
this insight: It tries to find solutions that maximize the customer
demands while staying within a small distance of $\tilde{d}$. In so doing,
it restricts attention to input vectors in the set ${\cal N}$ defined
as
\[
{\cal N} = \{ d \mid \exists x \geq 0: g(x,d) \geq 0 \mbox{ and }  |f(x) - \tilde{f} \| \leq \beta \}.
\]


\paragraph{Solution Method} 

This section presents an effective algorithm for solving the bilevel
optimization problem \eqref{eq:BL} when \eqref{OPT} is monotone with respect to
the sensitive data, capturing the above intuition. 

\begin{definition}[Monotonicity]
\label{def:monotonic}
\eqref{OPT} is monotone if there exists a proxy concave function $m: \Re^m
\rightarrow \Re$ such that if, for all $d^1, d^2 \in {\cal N}$,
$
     m(d^1) \geq m(d^2) \Rightarrow {\cal O}(d^1) \geq {\cal O}(d^2).
$
\end{definition}

\noindent
Note that the monotonicity assumption applies only to input vectors in
${\cal N}$.
While monotonicity is unlikely to be true for general parametric optimization problems ${\cal O}$, since they can be nonconvex, the computational results reported in Section \ref{sec:experiments}
illustrates that this property holds on several real and realistic benchmarks.

To solve bilevel optimization problems over a monotone follower, the
approach relies on solving optimization problems of the form
\begin{flalign}
{\cal O}^{\uparrow}(\delta) = \;\;  & \max_{d,x \geq 0} \hspace{0.85cm} \;\;   m(d) & \tag{PP1} \\
&  \text{subject to } \;\; g(x,d) \geq 0 & \tag{PP2} \\ 
&  \hspace{1.4cm}  \;\; |f(x) - \tilde{f}| \leq \beta & \tag{PP3} \\
&  \hspace{1.4cm}  \;\; \| d - \tilde{d} \|^2_2 \leq  \ \delta & \tag{PP4}
\end{flalign}
\noindent
The set of optimal and feasible solutions to ${\cal
  O}^{\uparrow}(\delta)$ are denoted by ${\cal S}^{\uparrow}(\delta)$
and ${\cal F}^{\uparrow}(\delta)$ respectively.  For a given $\delta$,
${\cal O}^{\uparrow}(\delta)$ finds a vector $\bar{d} \in {\cal N}$
that maximizes $m(\cdot)$ but remains within a distance $\delta$ of
$\tilde{d}$ . Since, by monotonicity, $m(\cdot)$ is a proxy for maximizing
 the optimal cost for the optimization problem
${\cal O}(\cdot)$, the optimization can be viewed as searching for a
feasible solution of \eqref{eq:BL} within a distance $\delta$ of $\tilde{d}$
which maximizes ${\cal O}(\cdot)$. Figure \ref{fig:push_exp} illustrates
the role of ${\cal O}^{\uparrow}$. Vector $d_i \in {\cal N}$ is the
optimal solution of ${\cal O}^{\uparrow}(\delta_i)$. As the distance
$\delta_i$ increases, $d_i={\cal O}^{\uparrow}(\delta_i)$ and ${\cal O}(d_i)$
should increase as well.

Under the monotonicity assumption, \eqref{eq:BL} will be shown equivalent to
the following optimization problem:
\begin{flalign}
\min_{d,\delta \geq 0} \delta \mbox{ s.t. } [d \in S^{\uparrow}(\delta) \ \land \ {\cal O}(d) \geq \tilde{f} - \beta ] \tag{BLM} \label{eq:blm}
\end{flalign}
which is defined only in terms of ${\cal O}$ and ${\cal
  O}^{\uparrow}$. Observe that $\delta$ is a scalar and hence it is
natural to solve \eqref{eq:blm} using binary search as depicted in
Algorithm \eqref{alg:BLM}. The algorithm receives as input a tuple
$\langle \delta^{l}, \delta^{u}, \eta \rangle$, where $\delta^{l}$ and
$\delta^{u}$ are lower and upper bounds on the optimal value
$\delta^*$ for \eqref{eq:blm}, and produces an $\eta$-approximation to
\eqref{eq:blm}. Algorithm \eqref{alg:BLM} is a simple binary search on
the value $\delta$ alternating the optimizations of ${\cal O}$ and
${\cal O}^{\uparrow}$. For a given $\delta$, line 3 solves ${\cal
  O}^{\uparrow}$. If the resulting optimization satisfies the second
constraint of \eqref{eq:blm}, a new feasible solution to \eqref{eq:BL}
is obtained and the upper bound can be updated. Otherwise, Algorithm
\eqref{alg:BLM} has identified a new lower bound.  Note that, in
practice, good lower and upper bounds are often available. For
instance, it is possible to use an optimal solution $d^h$ of
\eqref{eq:HPR} and set $\delta^l = \delta^h = \| d^h - \tilde{d}
\|^2_2$.  To obtain $\delta^{u}$, one can start with $\delta^{h}$ and
continue by doubling its value iteratively until Lemma
\ref{lemma:feasible} below applies.

\begin{algorithm}[!t]
\SetKwInOut{Input}{Inputs}
\SetKwInOut{Output}{Output}
\caption{Solving the BLM Optimization Problem.\!\!\!\!\!}
\label{alg:BLM}
\SetKwInOut{Input}{Inputs}
\SetKwInOut{Output}{Output}
\SetKwFunction{FR}{Phase1}
\SetKwFunction{FC}{Phase2}
\SetKwProg{Fn}{Function}{:}{}
\SetKwProg{Mn}{Algorithm}{:}{}

\Input{$\langle \delta^{l}, \delta^{u}, \eta \rangle$}
\Output{An $\eta$-approximation to \eqref{eq:BL}}
  \While{$\delta^u - \delta^l > \eta$}
  {
    $\delta \gets \frac{\delta^{l} + \delta^{u}}{2} $\\
    solve ${\cal O}^\uparrow(\delta)$ and let $(d^{\uparrow},x^{\uparrow})$ be an optimal solution \\
    $\delta^{\uparrow} \gets \| d^{\uparrow} - \tilde{d} \|$ \\
    \lIf{${\cal O}(d^{\uparrow}) > \tilde{f} - \beta$}{
	$\delta^u \gets \delta^{\uparrow} \mbox{ {\bf else} }$ $\delta^l \gets \delta$
    }
  }
  \KwRet ${\cal S}^{\uparrow}(\delta^u)$ 
\end{algorithm}

\paragraph{Correctness} It remains to prove the correctness of the solution technique. The following two lemmas capture important properties of ${\cal
  O}^{\uparrow}$. The first lemma shows that, when $\delta$ is large
enough, ${\cal O}^{\uparrow}(\delta)$ always returns a feasible
solution to \eqref{eq:BL}.

\begin{lemma}
\label{lemma:feasible}
Let $\delta^o = \| d^o - \tilde{d} \|_2^2 \leq \delta^{o}$ and $(d^{\uparrow},x^{\uparrow}) \in {\cal  S}^{\uparrow}(\delta^{o})$. Then $d^{\uparrow} \in {\cal F}^{BL}$.
\end{lemma}
\begin{proof}
  By definition of $\delta^o$, $d^o \in {\cal
    F}^{\uparrow}(\delta^o)$ and there exists $x^o$ such that the
  pair $(d^o,x^o)$ satisfies conditions (PP2)--(PP4).  Since
  $(d^{\uparrow},x^{\uparrow}) \in {\cal S}^{\uparrow}(\delta^{o})$,
  $m(d^{\uparrow}) \geq m(d^o)$ and, by monotonicity, ${\cal
    O}(d^{\uparrow}) \geq {\cal O}(d^{o}) \geq \tilde{f} - \beta$. By (PP3),
  $f(x^{\uparrow}) \leq f^o + \beta$ and, by (PP2), ${\cal
    O}(d^{\uparrow}) \leq f(x^{\uparrow})$. Hence
  \[
  \tilde{f} + \beta \geq f(x^{\uparrow}) \geq {\cal O}(d^{\uparrow}) \geq {\cal O}(d^{o}) \geq \tilde{f} - \beta
  \]
  and  $d^{\uparrow}$ satisfies (BL2). 
\end{proof}

The second lemma shows that there is no feasible solution to \eqref{eq:BL}
within distance $\dot{\delta}$ of $\tilde{d}$ when ${\cal
  O}(\dot{\delta})$ returns a solution $d$ that violates (BL2).

\begin{lemma}
\label{lemma:infeasible}
Let $(\dot{d},\dot{x}) \in {\cal S}^{\uparrow}(\dot{\delta})$ and ${\cal O}(\dot{d}) < \tilde{f} - \beta$. Then,
$$
\forall d^f \in {\cal F}^{BL}: \| d^f - \tilde{d} \|_2^2 > \dot{\delta}.
$$
\end{lemma}
\begin{proof}
  Consider $d^f \in {\cal F}^{BL}$ and assume that $\| d^f - \tilde{d}
  \|_2^2 \leq \dot{\delta}$, which implies that $d^f \in {\cal
    F}^{\uparrow}(\dot{\delta})$. By optimality of $\dot{d}$,
  $m(\dot{d}) \geq m(d^f)$ and, by monotonicity, ${\cal O}(\dot{d})
  \geq {\cal O}(d^f)$. By (BL2), ${\cal O}(d^f) \geq \tilde{f} -
  \beta$ which contradicts ${\cal O}(\dot{d}) < \tilde{f} - \beta$.
\end{proof}

These two lemmas make it possible to prove the equivalence of \eqref{eq:BL} and
\eqref{eq:blm} when \eqref{OPT} is monotone.

\begin{theorem}
When (OPT) is monotone, \eqref{eq:BL} and \eqref{eq:blm} are equivalent.
\end{theorem}
\begin{proof}
  Let $(d^*,\delta^*)$ be the optimal solution to \eqref{eq:blm}. Such a
  solution always exists by Lemma \ref{lemma:feasible} and $\| d^* -
  \tilde{d} \|_2^2 \leq \delta^{*}$.  By definition of ${\cal
    O}^{\uparrow}$, ${\cal F}(d^*) \neq \emptyset$ and ${\cal O}(d^*)
  \leq \tilde{f} + \beta$. Since ${\cal O}(d^*) \geq \tilde{f} -
  \beta$ by definition of \eqref{eq:blm}, $d^* \in {\cal F}^{BL}$.

  Consider now $\delta^- < \delta^*$ and a solution $d^- \in {\cal
    S}^{\uparrow}(\delta^-)$. If such a solution exists, ${\cal
    O}(d^-) < f^o - \beta$ by optimality of $\delta^*$. By Lemma
  \ref{lemma:infeasible}, it comes that
$
\forall d^f \in {\cal F}^{BL}: \| d^f - \tilde{d} \|_2^2 > \delta^-.
$
Hence, $d^*$ is also optimal for \eqref{eq:BL}. 
\end{proof}

It remains to show that the algorithm computes an $\eta$-approximation. 
\begin{theorem} Algorithm \eqref{alg:BLM} computes an $\eta$-approximation
of \eqref{eq:BL} when \eqref{OPT} is monotone.
\end{theorem}
\begin{proof}
Let $d^*$ be an optimal solution to \eqref{eq:BL}. Upon termination of Algorithm \eqref{alg:BLM}, it comes
that 
$
\delta^l \leq \| d^* - \tilde{d} \|^2_2 \leq \delta^u.
$
Moreover, if $d^u \in S^{\uparrow}(\delta^u)$, it follows that $\| d^u - \tilde{d} \|^2_2 \leq \delta^u$. Hence, 
$
\delta^l \leq \| d^* - \tilde{d} \|^2_2 \leq \| d^u - \tilde{d} \|^2_2 \leq \delta^u.
$
\end{proof}
\paragraph{Discussion}
In general, bilevel optimization is extremely difficult and general techniques and reformulations are often challenging to solve. This paper isolates
a class of bilevel problems where the optimal cost of the follower subproblem can be inferred, through monotonicity, by a proxy function that depends on variables controlling also the costs of the leader problem  \eqref{BL1}.
In game-theoretical settings, the proxy function serves as a bridge between the 
leader decisions to the costs of the follower. The existence of such a monotone function $m$ depends on the privacy application but is natural in energy systems where the customer load/demand must be protected. Indeed, the load appears linearly in flow balance constraints and load increases almost always lead to cost increases. 
Exploring other classes of problems (e.g. bounded monotonicity) will be left as future work.

\section{Applications on Energy Systems}

\begin{model}[!t]
	{\small
	\caption{${\cal O}_{\text{OPF}}$: AC Optimal Power Flow}
	\label{model:ac_opf}
	\vspace{-6pt}
	\begin{align}
		\mbox{\bf variables:} \;\;
		& S^g_i, V_i \;\; \forall i\in N, \;\;
		  S_{ij} 	 \;\; \forall(i,j)\in E \cup E^R \nonumber \\
		\mbox{\bf minimize:} \;\;
		& {\cO}(\bm{S^d}) = \sum_{i \in N} {c}_{2i} (\Re(S^g_i))^2 + {c}_{1i}\Re(S^g_i) + {c}_{0i} \label{ac_obj} \\
		\mbox{\bf subject to:} \;\; 
		& \angle V_{i} = 0, \;\; i = \min N \label{eq:ac_0} \\
		& {v}^l_i \leq |V_i| \leq {v}^u_i 		\;\; \forall i \in N \label{eq:ac_1} \\
		& {\theta}^{l}_{ij} \leq \angle (V_i V^*_j) \leq {\theta}^{u}_{ij} \;\; \forall (i,j) \in E  \label{eq:ac_2}  \\
		& {S}^{gl}_i \leq S^g_i \leq {S}^{gu}_i \;\; \forall i \in N \label{eq:ac_3}  \\
		& |S_{ij}| \leq {s}^u_{ij} 					\;\; \forall (i,j) \in E \cup E^R \label{eq:ac_4}  \\
		& S^g_i - {S}^d_i = \textstyle\sum_{(i,j)\in E \cup E^R} S_{ij} \;\; \forall i\in N \label{eq:ac_5}  \\ 
		& S_{ij} = {Y}^*_{ij} |V_i|^2 - {Y}^*_{ij} V_i V^*_j 			 \;\; \forall (i,j)\in E \cup E^R \label{eq:ac_6}
	\end{align}
	}
	\vspace{-12pt}
\end{model}


This section describes two substantial case studies for evaluating the
privacy mechanism: optimal power flow in electricity networks and
optimal compressor optimization in gas networks. Both models are nonlinear and nonconvex.

\paragraph{Optimal Power Flow} \emph{Optimal Power Flow (OPF)} is
the problem of finding the best generator dispatch to meet the demands
in a power network.  A power network $\bm{\cN}$ can be represented as
a graph $(N, E)$, where the nodes in $N$ represent buses and the edges
in $E$ represent lines. The edges in $E$ are directed and $E^R$ is
used to denote those arcs in $E$ but in reverse direction.  The AC
power flow equations are based on complex quantities for current $I$,
voltage $V$, admittance $Y$, and power $S$, and these equations are a
core building block in many power system applications.
Model~\ref{model:ac_opf} shows the AC OPF formulation, with
variables/quantities shown in the complex domain.  Superscripts $u$ and
$l$ are used to indicate upper and lower bounds for variables. The
objective function ${\cO}(\bm{S^g})$ captures the cost of the
generator dispatch, with $\bm{S^g}$ denoting the vector of generator
dispatch values $(S^g_i \:|\: i \in N)$.  
Constraint \eqref{eq:ac_0} sets the reference voltage angle to zero for the 
slack bus $i \in N$ to eliminate numerical symmetries.  
Constraints \eqref{eq:ac_1} bound the voltage magnitudes,
and constraints \eqref{eq:ac_2} limit the voltage angle differences 
for every transmission lines/transformers.
Constraints \eqref{eq:ac_3} enforce the generator output limits,
and constraints \eqref{eq:ac_4} impose the line flow limits.
Finally, constraints \eqref{eq:ac_5}
capture Kirchhoff's Current Law imposing the flow balance across every node,
and constraints \eqref{eq:ac_6} capture Ohm's Law describing 
the AC power flow across lines/transformers.

\begin{model}[!t]
	{\small
	\caption{${\cal O}_{\text{OGF}}$: Optimal Gas Flow}
	\label{model:ac_ogf}
	\vspace{-6pt}
	\begin{align}
		\mbox{\bf variables:} \;\;
		& p_i, q_i \; \forall i \in \cJ, \;
		  q_{ij} \; \forall (i,j) \in \cP, \;
		  R_{ij} \; \forall(i,j) \in \cC \nonumber \\
		\mbox{\bf minimize:} \;\;
		& {\cO}(\bm{q}) = \sum_{(i,j) \in C} \mu^{-1}|q_{ij}|(\max\{R_{ij},1\}^{2(\gamma-1)/\gamma}-1) \label{gas_obj} \\
		\mbox{\bf subject to:} \;\; 
		& \textstyle\sum_{(i,j) \in \cP} q_{ij} - \sum_{(j,i) \in \cP} q_{ji} = q_i, \;\; \forall i \in \cJ \label{eq:gas_0} \\
		& {p_i}^l \leq p_i \leq {p_i}^u 		\;\; \forall i \in N, \;\; {q_{ij}}^l \leq q_{ij} \leq {q_{ij}}^u 		\;\; \forall (i,j) \in \cP \label{eq:gas_1} \\
		& {R_{ij}}^l \leq R_{ij} \leq {R_{ij}}^u 		\;\; \forall (i,j) \in \cC \label{eq:gas_2}\\
		& p_i = p^T_i 	\;\; \forall i \in \cJ^B, q_i = 0 	\;\; \forall i \in \cJ^T, q_i = q^d_i \;\; i  \in \cJ^D\label{eq:gas_4}\\
		& R_{ij}^2 p_{i}^2 - p_{j}^2 = L_{ij} \displaystyle \frac{\lambda a^2}{D_{ij} A_{ij}^2} q_{ij} \lvert q_{ij} \rvert		 \;\; \forall (i,j)\in \cC \label{eq:gas_5}\\
		& p_{i}^2 - p_{j}^2 = L_{ij} \displaystyle \frac{\lambda a^2}{D_{ij} A_{ij}^2} q_{ij} \lvert q_{ij} \rvert		 \;\; \forall (i,j)\in \cP - \cC \label{eq:gas_6}
	\end{align}
	}
	\vspace{-8pt}
\end{model}

\paragraph{Optimal Compressor Optimization}

\emph{Optimal Gas Flow (OGF)} is the problem of finding the best
compression control to maintain pressure requirements in a natural gas
pipeline system.  A natural gas network can be represented as a
directed graph $\bm{\cN}=(\cJ,\cP)$, where a node $i \in \cJ$
represents a junction point and an edge $(i,j)
\in \cP$ represents a pipeline.  Compressors ($\cC \subseteq \cP$) are
installed in a subset of the pipelines for boosting the gas pressure
$p$ in order to maintain pressure requirements for gas flow $q$.  The
set $\cJ^D$ of gas demands and the set $\cJ^T$ of transporting nodes
are modeled as junction points, with net gas flow $q_i$ set to the gas
demand ($q^d_i$) and zero respectively.  For simplicity, the paper
assumes no pressure regulation and losses within junction nodes and
gas flow/flux are conserved throughout the system.  A subset $\cJ^B
\in \cJ$ of the nodes are regulated with constant pressure $p^T_i$.
The length of pipe $(i,j)$ is denoted by $L_{ij}$, its diameter by
$D_{ij}$, and its cross-sectional area by $A_{ij}$.  Universal
quantities include isentropic coefficient $\gamma$, compressor
efficiency factor $\mu$, sound speed $a$, and gas friction factor
$\lambda$.  Model~\ref{model:ac_ogf} depicts the OGF formulation.  The
objective function ${\cO}(\bm{q})$ captures the compressor costs using
the compressor control values $(R_{ij} \:|\: (i,j) \in \cC)$.
Constraints \eqref{eq:gas_0} capture the flow conversation equations.
Constraints \eqref{eq:gas_1} and \eqref{eq:gas_2} capture the
pressure, flux flow, and compressor control bounds.  Constraints
\eqref{eq:gas_4} set the boundary conditions for the demands and the
regulated pressures. Finally, constraints \eqref{eq:gas_5} and
\eqref{eq:gas_6} capture the steady-state isothermal gas flow
equation that describes the pressure loss mechanics within gas pipes.

\begin{figure}[t]
	\centering
	\includegraphics[width=.40\textwidth]{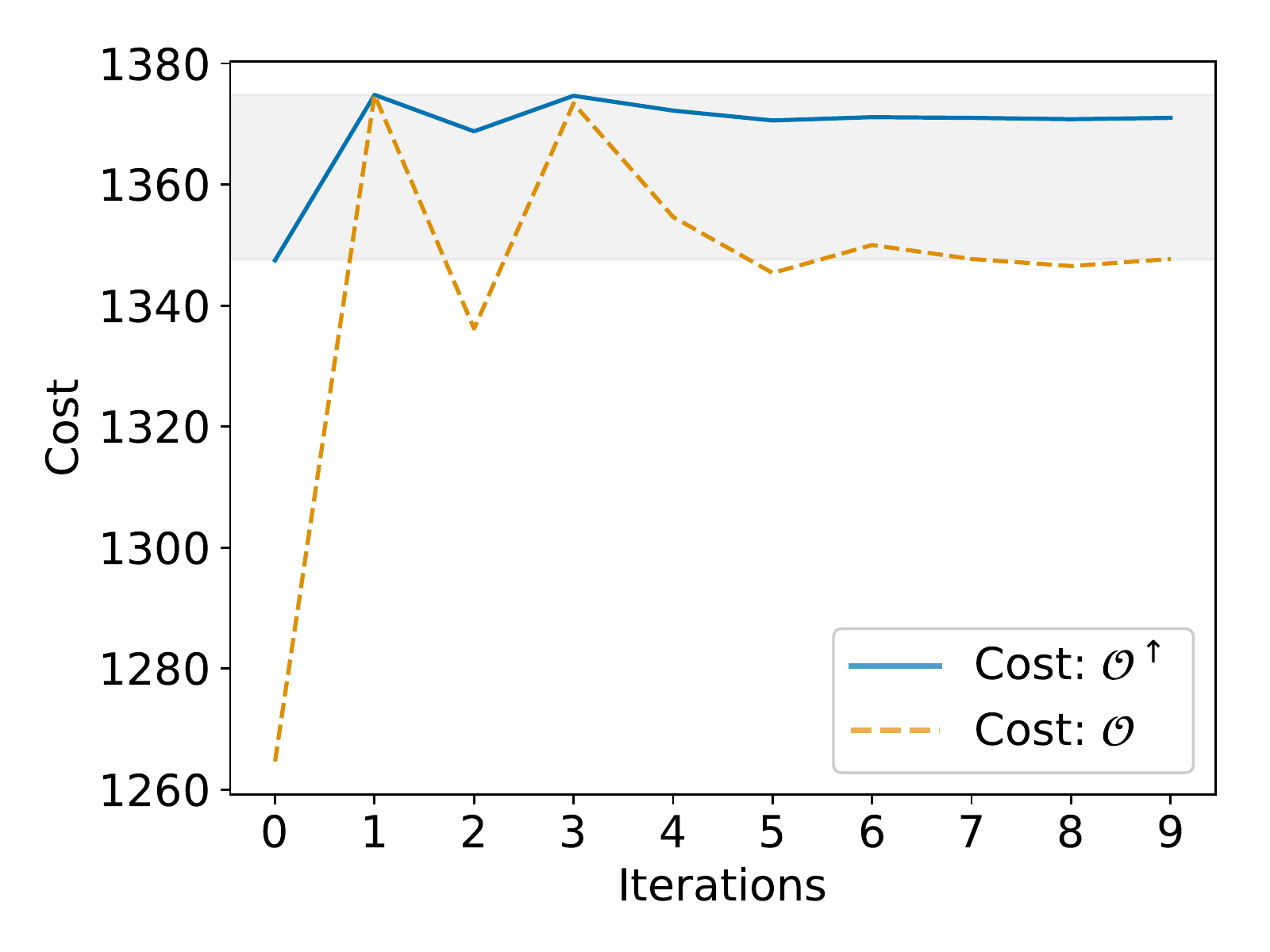}
	\includegraphics[width=.40\textwidth]{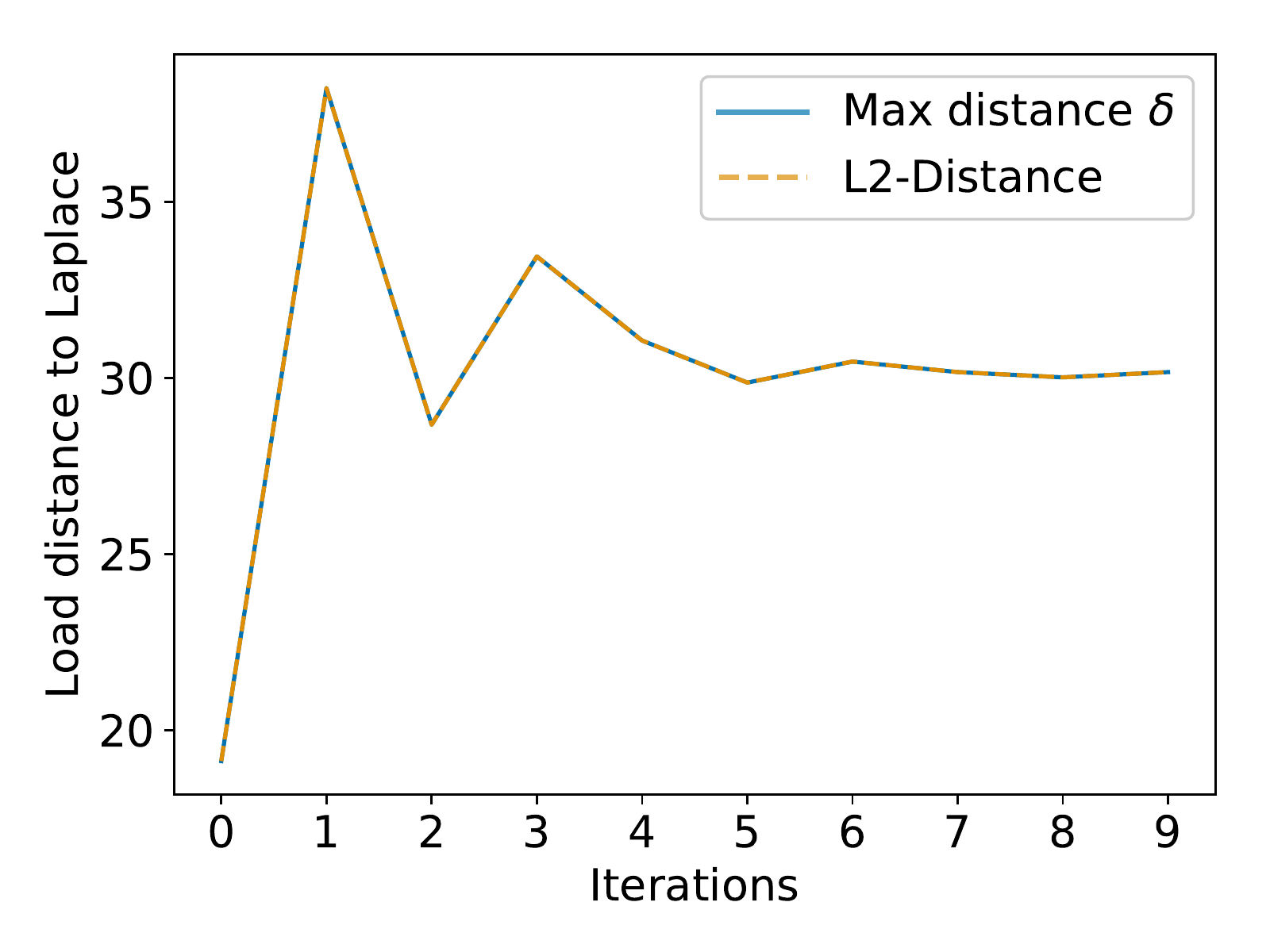}
	\caption{Gaslib-135: OGF costs and L$^2$-distance to $\tilde{\delta}$ as a function of the number of iterations for parameters $\alpha = 0.1$ and $\beta=1\%$.} 
	\label{fig:convergence_example}
\end{figure}
\paragraph{Obfuscation}
In both networks, customer demands are sensitive and are associated
with customer activities. These quantities always appear in the flow
conservation constraints, which are linear. Increasing these values
obviously requires more (electricity and gas) production and hence one
can expect the cost to increase. This is not always the case, because
of the engineering constraints on voltages and pressures and the lower
bounds of the production units, for instance. However, in practice,
because of reliability constraints, these constraints are rarely
binding at optimality and one may expect the systems to behave
monotonically around the optimality point. The above is validated in
the experimental results.


\section{Experimental Evaluation}
\label{sec:experiments}
\paragraph{Setting}
The experiments were performed on a variety of NESTA power
systems~\cite{Coffrin14Nesta} and natural gas test systems
from~\cite{mak19dynamic}, including test instances from
GasLib~\cite{pfetsch2015validation}. Parameter $\epsilon$ is fixed to
1.0 and the \emph{indistinguishability level} $\alpha$ varies from
$10^{-1}$ to $10^{1}$ (in per unit notation).  The \emph{fidelity}
parameter $\beta$ varies from 0.1\% to 10\% of value $\tilde{f} = f^o$.  Convergence parameter $\eta$ is set to $10^{-3}$ (in per unit). The lower bounds and upper bounds are initialized as
specified in prior sections. The solution technique is limited to 3000
calls to ${\cal O}^{\uparrow}$. The models are implemented with
PowerModels.jl~\cite{Coffrin:18} with the nonlinear solver
IPOPT~\cite{wachter06on}. Note that some of the test cases have more
than $10^4$ variables.

\paragraph{Behavior of the Solution Technique}
Figure~\ref{fig:convergence_example} shows how the costs (${\cal O}$
and ${\cal O}^{\uparrow}$) and L$^2$-Distance to $\tilde{\delta}$
typically change when running Algorithm~\ref{alg:BLM} and its
initialization.  The L$^2$-Distance to $\tilde{d}$ increases initially
to find a solution within the feasible cost range (shaded area). The
binary search then finds the optimal distance in a few
iterations. (PP4) is usually binding when optimizing ${\cal
  O}^{\uparrow}$.

\begin{table*}[t]
\centering
\caption{OPF: Average CPU times and number of optimizations (50 runs) for parameters $\alpha = \{0.1, 1.0, 10.0\}$ and $\beta=1\%$.}
\resizebox{.63\textwidth}{!}{%
\begin{tabular}{r|ccc|ccc}
\toprule
$\alpha$ & $0.1$     & $1.0$ & $10.0$ & $0.1$     & $1.0$ & $10.0$ \\
\midrule
Benchmark &  Time (s) & Time (s) & Time (s) & Opt. & Opt. & Opt. \\
\midrule
case14\_ieee      &   0.71 &   0.58 & 1.01 & 10.22 & 4.40 & 6.10\\
case24\_ieee\_rts &   1.52 &   1.40 & 2.30 & 8.90 & 4.14 & 5.10\\
case29\_edin      &   0.67 &   36.84 & 23.50 & 1.00 & 40.36 & 27.78\\
case30\_as        &   0.90 &   0.90 & 0.85 & 4.62 & 4.28 & 3.80\\
case30\_fsr       &   1.28 &   1.07 & 0.66 & 6.24 & 4.70 & 2.82\\
case30\_ieee      &   0.84 &   1.22 & 1.33 & 5.14 & 5.86 & 6.06\\
case39\_epri      &   1.16 &   6.92 & 1.29 & 6.32 & 21.50 & 3.14\\
case57\_ieee      &   1.77 &   4.22 & 10.31 & 4.98 & 3.72 & 7.94\\
case73\_ieee\_rts  &  6.20 &   2.17 & 11.57$^{1}$ & 9.88 & 1.62 & 66.54\\
case89\_pegase    &   9.23 &   13.56 & 17.60 & 6.04 & 5.26 & 7.82\\
case118\_ieee     &   10.15&   10.71 & 9.77 & 10.08 & 5.96 & 6.42\\
case162\_ieee\_dtc &  10.04 &  35.18 & 53.72 & 4.88 & 9.68 & 8.82\\
case189\_edin     &   2.98 &   53.40 & 58.08 & 2.16 & 7.98 &8.18\\
case240\_wecc     &   12.34 & 117.78 & 148.16 & 1.16 & 10.88 & 10.10\\
case300\_ieee     &   1621.78& 360.99 & 301.32 & 8.90 & 7.98 & 10.32\\
case1354\_pegase  &   15.90 & 2172.23 & 2363.19 & 1.16 & 7.68 & 9.10\\
case1394sop\_eir  &   70.78 &  777.66 &  1427.87 & 3.72 & 6.12 & 5.40\\
case1397sp\_eir   &   146.47 & 1298.15 & 660.12 & 3.90 & 6.40 & 4.60\\
case1460wp\_eir   &   157.06 & 1418.58 & 728.84 & 4.72 & 6.90 & 5.66\\
\bottomrule
\end{tabular}
}
\label{tbl:power-time}
\end{table*}
\begin{table*}[t]
\centering
\caption{OGF: Average CPU times and number of optimizations (50 runs) for parameters $\alpha = \{0.1, 1.0, 2.0\}$ and $\beta=1\%$.}
\resizebox{.55\textwidth}{!}{%
\begin{tabular}{r|ccc|ccc}
\toprule
$\alpha$ & $0.1$     & $1.0$ & $2.0$ & $0.1$     & $1.0$ & $2.0$ \\
\midrule
Benchmark &  Time (s) & Time (s) & Time (s) & Opt. & Opt. & Opt. \\
\midrule
24-pipe    &      2.34 & 0.67 & 1.77 & 20.56 & 6.20 & 6.22\\
GasLib-40  &     34.74 & 141.16 & 132.65 & 26.50 & 33.48 & 28.82\\
GasLib-135 &     21.63 & 28.91 & 214.57 & 12.96 & 8.74 & 25.80\\
\bottomrule
\end{tabular}
}
\label{tbl:gas-time}
\end{table*}
\begin{table*}[t]
\centering
\caption{OPF test cases: Average OPF cost difference (in \%) for parameters $\alpha = 0.1$ and $\beta=\{10\%, 1\%, 0.1\%\}$.}
\resizebox{.50\textwidth}{!}{%
\begin{tabular}{r|ccc|ccc}
\toprule
 &  \multicolumn{3}{c|}{Bi-level} &  \multicolumn{3}{c}{High-point}\\
\midrule
$\beta (\%)$ & $10\%$     & $1\%$ & $0.1\%$ & $10\%$     & $1\%$ & $0.1\%$ \\
\midrule
case14\_ieee      &       0.1 &       0.1 &       0.0 &      -4.9 &      \bf{-6.2} &      \bf{-6.3} \\
case24\_ieee\_rts  &       0.0 &       0.4 &       0.1 &      -0.1 &      \bf{-1.6} &      \bf{-1.9} \\
case29\_edin      &      -0.0 &      -0.0 &       0.0 &      -0.0 &      -0.0 &      -0.1 \\
case30\_as        &       0.7 &       0.3 &      -0.1 &      -1.1 &      \bf{-2.3} &      \bf{-2.4} \\
case30\_fsr       &      -0.1 &       0.3 &      -0.1 &      -6.2 &      \bf{-6.8} &      \bf{-6.9} \\
case30\_ieee      &      -0.3 &       0.1 &       0.0 &      -2.7 &      \bf{-1.8} &      \bf{-1.6} \\
case39\_epri      &      -0.2 &      -0.0 &       0.0 &      -0.2 &      -0.4 &      \bf{-0.7} \\
case57\_ieee      &       1.5 &       0.4 &       0.1 &       1.4 &      -0.6 &      \bf{-1.0} \\
case73\_ieee\_rts  &      -0.4 &       0.1 &       0.0 &      -0.4 &      \bf{-1.1} &      \bf{-1.4} \\
case89\_pegase    &      -0.3 &      -0.1 &       0.0 &      -0.3 &      -0.4 &      \bf{-0.6} \\
case118\_ieee     &      -0.0 &       0.3 &       0.0 &      -0.0 &      \bf{-1.6} &      \bf{-2.0} \\
case162\_ieee\_dtc &       2.6 &       0.5 &       0.1 &       2.3 &      -0.1 &      \bf{-0.7} \\
case189\_edin     &       9.2 &       1.0 &       0.1 &       9.2 &       0.8 &      -0.1 \\
case240\_wecc     &       0.4 &       0.1 &       0.0 &       0.4 &       0.1 &      \bf{-0.2} \\
case300\_ieee     &       8.9 &       0.9 &       0.1 &       8.8 &       0.9 &       0.1 \\
case1354\_pegase  &       0.0 &       0.0 &       0.0 &       0.0 &       0.0 &      -0.1 \\
case1394sop\_eir  &       9.8 &       1.0 &       0.1 &      \bf{10.3} &       \bf{1.2} &       \bf{0.3} \\
case1397sp\_eir   &      10.0 &       1.0 &       0.1 &      \bf{10.5} &       \bf{2.0} &       \bf{0.7} \\
case1460wp\_eir   &       9.8 &       1.0 &       0.1 &      \bf{14.6} &       \bf{1.5} &       \bf{0.7} \\
\bottomrule
\end{tabular}
}
\label{tbl:power-opf}
\end{table*}

\begin{figure}[t]
	\centering
	\includegraphics[width=.40\textwidth]{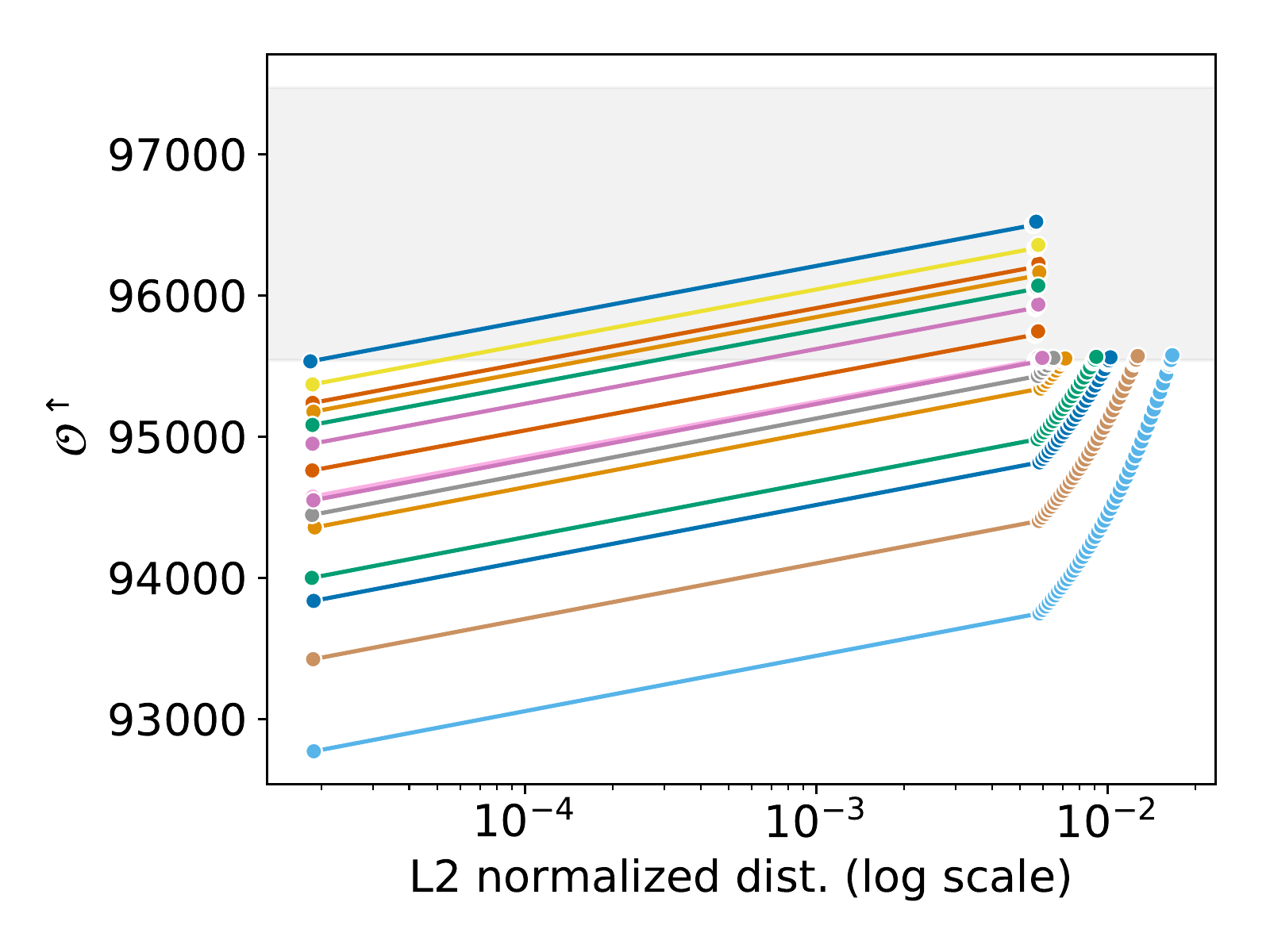}
	\includegraphics[width=.40\textwidth]{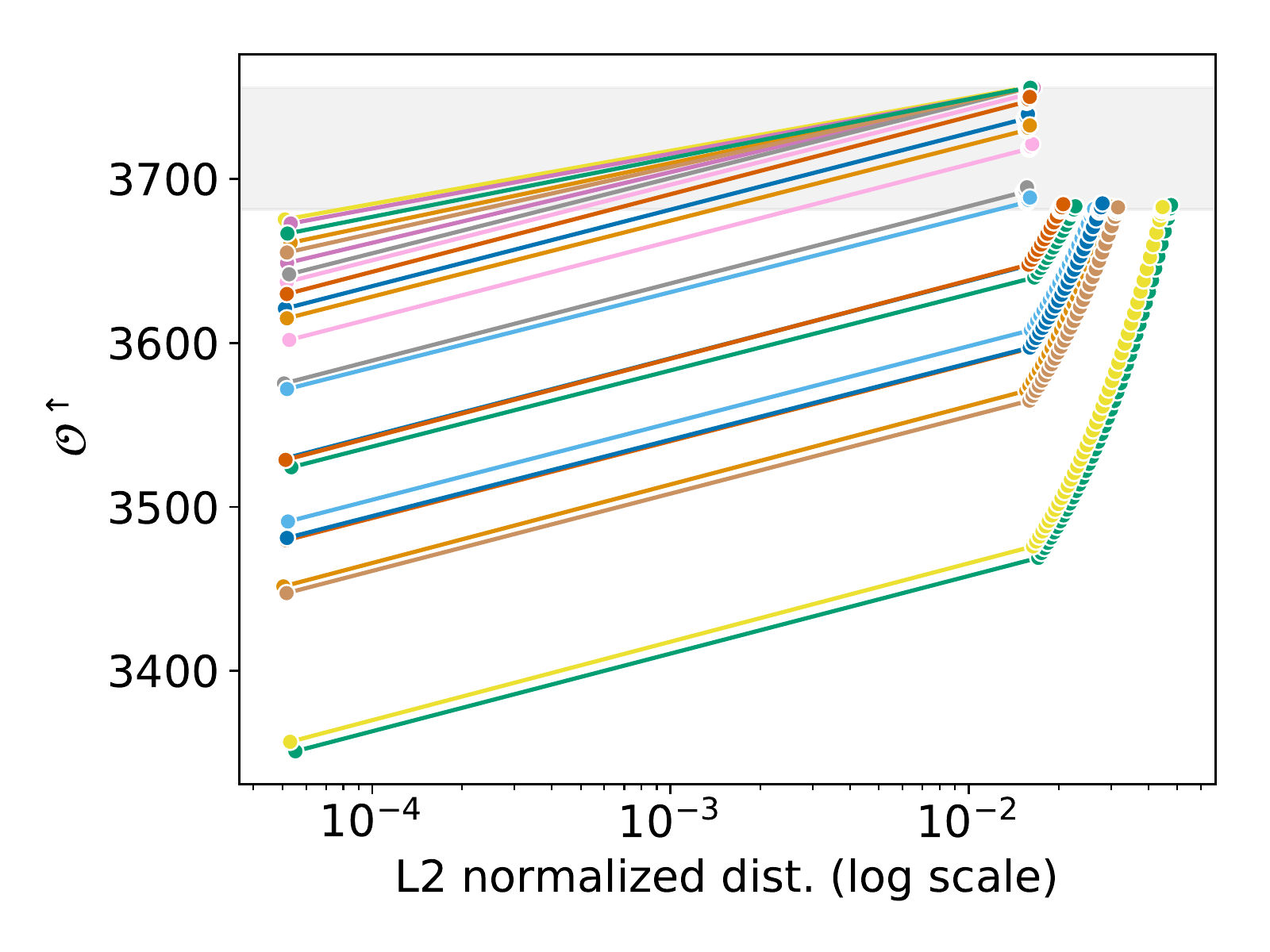}
	\caption{IEEE-39 (left) IEEE-118 (right): OPF cost as a function of the L$^2$-distance to $\tilde{\delta}$ for parameters $\alpha = 0.1$ and $\beta= 1\%$.} 
	\label{fig:power-mono}
\end{figure}
\begin{figure}[t]
	\centering
	\includegraphics[width=.40\textwidth]{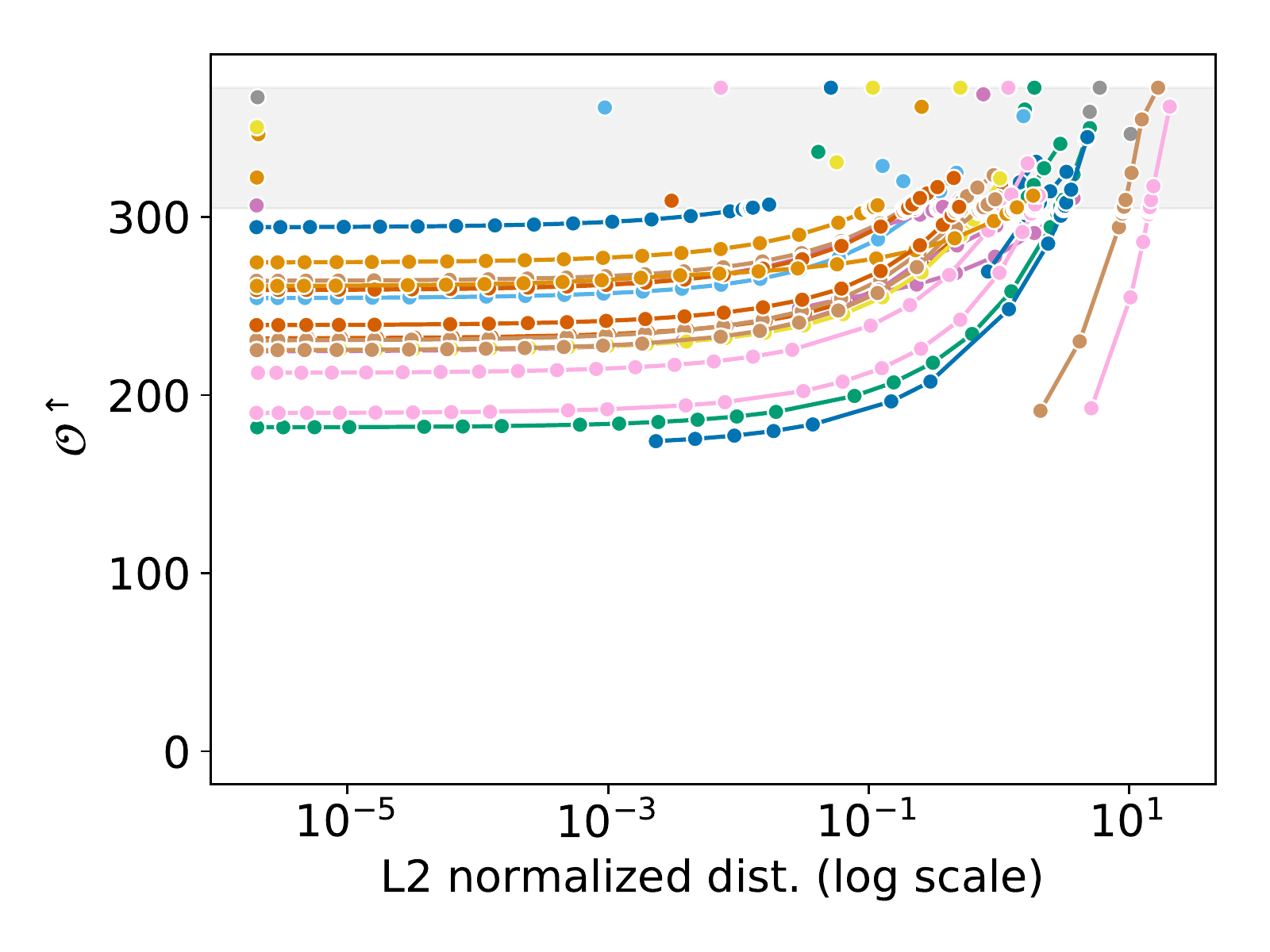}
	\includegraphics[width=.40\textwidth]{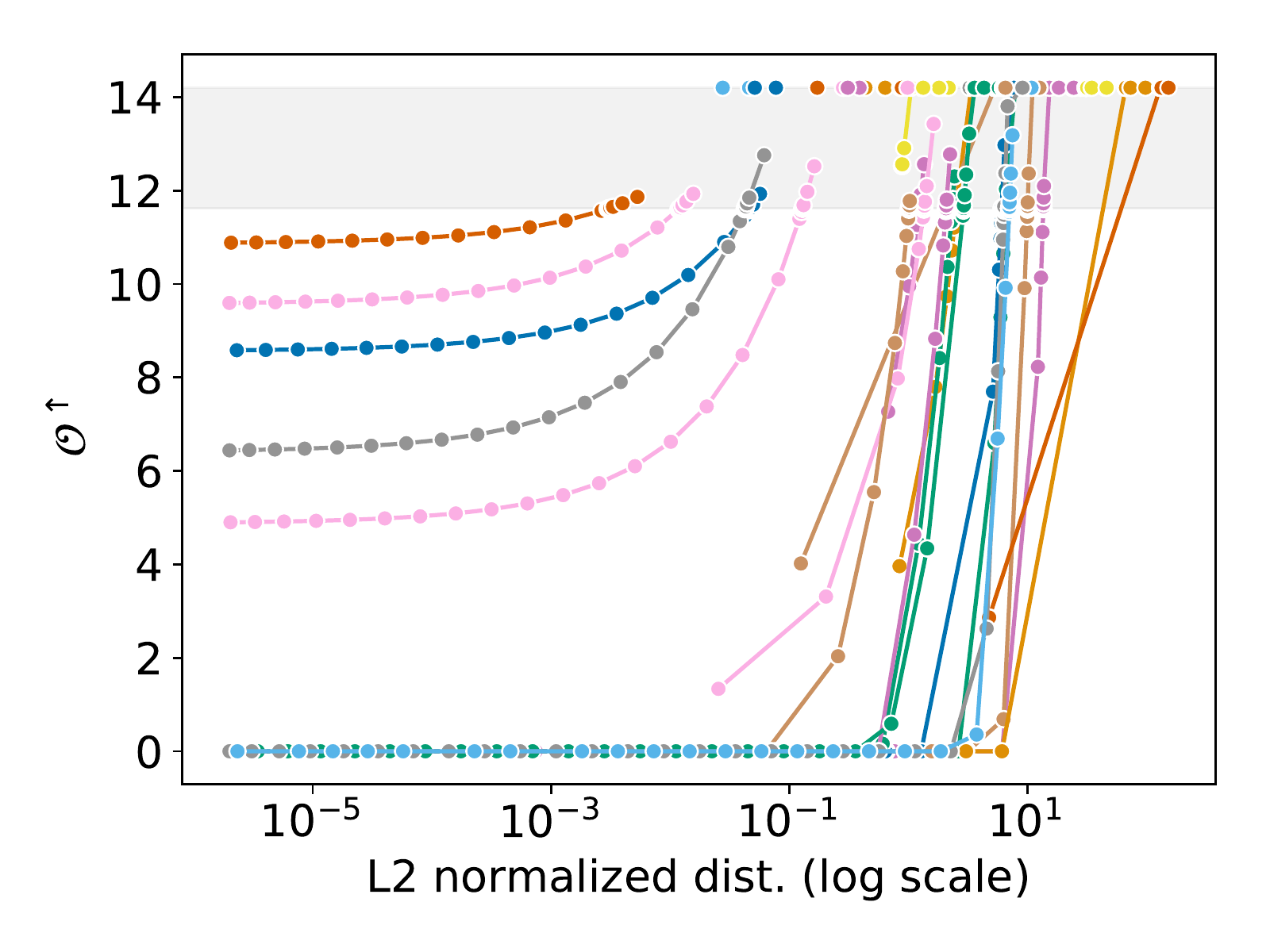}
	\caption{24-pipe (left) Gaslib-40 (right): OGF cost as a function of the L$^2$-distance to $\tilde{\delta}$ for parameters $\alpha = 0.1$ and $\beta= 10\%$.} 
	\label{fig:gas-pipe-mono}
\end{figure}

\paragraph{Convergence and Computational Efficiency}

The experimental results demonstrate the robustness and scalability of
the approach.  Tables~\ref{tbl:power-time} and~\ref{tbl:gas-time}
depict the average CPU times (in seconds) and average number of calls
to ${\cal O}^{\uparrow}$ (and thus ${\cal O}$) also over 50 runs. All
instances (except one\footnote{The failure to converge on 
 IEEE-73 (marked with a superscript) is due to a
  poor starting point from the HPR.})
converge with a very small number of iterations. Even large-scale test
cases with more than $10^4$ variables are solved in a few iterations.
\emph{The bilevel model is thus a truly practical approach to demand
  obfuscation of these networks}.


\paragraph{Monotonicity}

Figures~\ref{fig:power-mono} and~\ref{fig:gas-pipe-mono} illustrate
the interpolation of the optimal values ${\cal O}^{\uparrow}$ (y-axis) with the respect to their distances to $\tilde{\delta}$ (x-axis) obtained while solving the bilevel model. The shaded area shows the feasible cost range. The results summarize 50 runs (each with a different random seed), each represented by a colored curve. \emph{As can be
  seen, the monotonicity property holds in these real networks}. Note
also the single points in the gas plot: These are cases where the
high-point relaxation can be made optimal in one iteration.

\paragraph{The Benefits of the Bilevel Model}

Tables \ref{tbl:power-opf} and \ref{tbl:gas-opf} show the benefits of
the bilevel model compared to the HPR. The HPR returns a solution
$d^h$ in ${\cal N}$ with the smallest distance to $\tilde{d}$.
However, it is typically the case that ${\cal O}(d^h) < \tilde{f} -
\beta$. The tables report the relative distance of ${\cal O}(d^*)$ and
${\cal O}(d^h)$ to  $\tilde{f}$ for various values of $\beta$.
Numbers are bold for instances violating the $\beta$ thresholds.
Table \ref{tbl:gas-opf} shows that, on the gas networks, the bilevel model
produces several orders of magnitude improvements over the HPR. The
gains are less pronounced on the electricity systems, but remain
substantial. Note that obfuscation methods with a few percents of cost difference mean
an obfuscation error of hundreds of millions of dollars in energy systems.


\begin{table*}[h]
\centering
\caption{OGF test cases: Average OGF cost differences (in \%) for parameters $\alpha = 0.1$ and $\beta=\{10\%, 1\%, 0.1\%\}$.}
\resizebox{.45\textwidth}{!}{%
\begin{tabular}{r|ccc|ccc}
\toprule
 &  \multicolumn{3}{c|}{Bi-level} &  \multicolumn{3}{c}{High-point}\\
\midrule
$\beta (\%)$ & $10\%$     & $1\%$ & $0.1\%$ & $10\%$     & $1\%$ & $0.1\%$ \\
\midrule
24-pipe    &-3.1&-0.3&    0.0 &\bf{-12.1}&\bf{-13.7}&   \bf{-14.1}\\
gaslib-40  &2.5&0.3&     0.0 &\bf{-31.5}&\bf{-38.1}&   \bf{-39.0}\\
gaslib-135 &1.3&0.2&     0.0 &\bf{-37.3}&\bf{-38.4}&   \bf{-41.6}\\
\bottomrule
\end{tabular}
}
\label{tbl:gas-opf}
\end{table*}

\begin{table*}[!h]
\centering
\caption{OPF test cases: Average distance (L2, normalized to original scale) between obfuscated and original solutions for parameters $\alpha = \{0.1,1.0,10.0\}$ and $\beta=1\%$.}
\resizebox{.70\textwidth}{!}{%
\begin{tabular}{r|ccc|ccc|ccc}
\toprule
 &  \multicolumn{3}{c|}{Bi-level} &  \multicolumn{3}{c|}{High-point} & \multicolumn{3}{c}{Laplace}\\
\midrule
$\alpha$ (p.u.) & $0.1$     & $1.0$ & $10.0$ & $0.1$     & $1.0$ & $10.0$ & $0.1$     & $1.0$ & $10.0$\\
\midrule
case14\_ieee      &0.67&     4.48 &    10.32 &0.67&    4.51 &    10.60 &    0.71 &    7.08 &    70.82 \\
case24\_ieee\_rts  &0.13&     1.11 &     3.61 &0.13&    1.11 &    3.63 &    0.13 &    1.34 &    13.39 \\
case29\_edin      &0.01&     0.09 &     0.82 &0.01&    0.09 &    0.82 &    0.01 &    0.09 &    0.87 \\
case30\_as        &0.84&     3.34 &     4.45 &0.84&    3.35 &    4.49 &    0.97 &    9.74 &    97.43 \\
case30\_fsr       &1.42&     5.47 &     7.40 &1.43&    5.49 &    7.47 &    1.66 &    16.61 &   166.14 \\
case30\_ieee      &0.88&     3.86 &     5.46 &0.88&    3.88 &    5.60 &    0.97 &    9.74 &    97.43 \\
case39\_epri      &0.06&     0.60 &     2.44 &0.06&    0.60 &    2.44 &    0.06 &    0.62 &    6.24 \\
case57\_ieee      &0.34&     2.38 &     6.64 &0.34&    2.42 &    7.56 &    0.35 &    3.51 &    35.14 \\
case73\_ieee\_rts  &0.13&     1.11 &     3.65 &0.13&    1.11 &    3.67 &    0.13 &    1.32 &    13.19 \\
case89\_pegase    &0.06&     0.53 &     2.64 &0.06&    0.53 &    2.86 &    0.06 &    0.56 &    5.60 \\
case118\_ieee     &0.40&     3.03 &     6.85 &0.40&    3.04 &    6.87 &    0.40 &    3.99 &    39.89 \\
case162\_ieee\_dtc &0.09&     0.68 &     1.59 &0.10&    0.69 &    1.61 &    0.10 &    0.96 &    9.62 \\
case189\_edin     &0.34&     2.05 &     3.92 &0.34&    2.07 &    4.18 &    0.35 &    3.50 &    35.05 \\
case240\_wecc     &0.01&     0.12 &     0.99 &0.01&    0.12 &    1.01 &    0.01 &    0.12 &    1.23 \\
case300\_ieee     &0.10&     0.85 &     3.69 &0.10&    0.86 &    4.46 &    0.11 &    1.07 &    10.65 \\
case1354\_pegase  &0.14&     1.26 &     5.47 &0.14&    1.26 &    6.21 &    0.14 &    1.36 &    13.59 \\
\bottomrule
\end{tabular}
}
\label{tbl:power-dist}
\end{table*}
\begin{table*}[!h]
\centering
\caption{OGF test cases: Average distance (L2, normalized to original scale) between obfuscated and original solutions for parameters  $\alpha = \{0.1,0.4,1.0\}$ and $\beta=1\%$.}
\resizebox{.60\textwidth}{!}{%
\begin{tabular}{r|ccc|ccc|ccc}
\toprule
 &  \multicolumn{3}{c|}{Bi-level} &  \multicolumn{3}{c|}{High-point} & \multicolumn{3}{c}{Laplace} \\
\midrule
$\alpha$ (p.u.) & $0.1$     & $0.4$ & $1.0$ & $0.1$     & $0.4$ & $1.0$ &  $0.1$     & $0.4$ & $1.0$ \\
\midrule
24-pipe    &          0.32 &    1.24& 2.68 &    0.34 & 1.29 & 2.71 &0.36  & 1.43 & 3.58\\
gaslib-40  &          2.04 &    7.66& 22.13 &    1.95 & 7.51& 16.54 &2.01 & 8.03 & 20.07\\
gaslib-135 &          0.72 &    2.78& 7.30 &    0.72 & 2.75 & 6.46 &0.72  & 2.88 & 7.19\\
\bottomrule
\end{tabular}
}
\label{tbl:gas-dist}
\end{table*}
\paragraph{Obfuscation Quality}

Tables~\ref{tbl:power-dist} and~\ref{tbl:gas-dist} report the
L$^2$-distances to the original loads for the high-point relaxation,
the bilevel model, and the load vector $\tilde{d}$ (which typically
produces infeasible problems).  The results are averaged over 50 runs.
The results show that the fidelity of the bilevel model (i.e., how
close $d^*$ is to $d^o$) is extremely high, and often improves over
the fidelity of the HPR. Finally, the bilevel model has a much higher
fidelity than the Laplace mechanism.

\section{Conclusion}

This paper presented a bilevel optimization model for postprocessing
the differentially private input of a constrained optimization
problem. The model restores the feasibility and near-optimality of the
optimization problem. The paper shows that the bilevel model can be
solved effectively under a natural monotonicity assumption by
alternating the solving of the follower problem and the solving of a
novel optimization model that maximizes a proxy of the true
objective. Experimental results on large-scale nonconvex constrained
optimization problems with more than $10^4$ variables demonstrate the
accuracy, efficiency, and benefits of the approach. They also validate
the monotonicity assumptions empirically. Future work will be devoted
to understanding and characterizing theoretically the solution space
around optimal solutions.

\section*{Acknowledgement}
The authors would like to thank Kory Hedman for extensive discussions on
various obfuscation techniques. This research is partly funded by the ARPA-E
Grid Data Program under Grant 1357-1530.

\newpage


\bibliographystyle{IEEEtran}
\bibliography{cpaior,opf_bib,differential_privacy}

\end{document}